\documentclass[12pt, a4paper]{article}
\usepackage{amssymb,amsmath,amsthm}
\usepackage{fancyhdr}
\usepackage[british]{babel}
\usepackage{enumitem}
\usepackage{multirow}
\usepackage{diagbox}
\usepackage{makecell}
\usepackage{colortbl} 
\usepackage{xcolor}
\usepackage{listings}
\usepackage{tikz}
\usetikzlibrary{decorations.pathreplacing}

\usepackage{algorithm}
\usepackage{algorithmic}
\usepackage{hyperref}
\usepackage[margin=2.5cm]{geometry}

\newtheorem{theorem}{Theorem}[section]

\newtheorem{lemma}[theorem]{Lemma}

\newtheorem{definition}[theorem]{Definition}
\theoremstyle{definition}
\newtheorem{case}{Case}

\newtheorem{example}{Example}
\newtheorem{observation}{Observation}
\newtheorem{remark}{Remark}

\newtheorem{corollary}[theorem]{Corollary}

\begin{document}
\title{\textbf{Rainbow and Gallai--Rado numbers involving binary function equations
		\footnote{Supported by NSFC No.12131013 and 12161141006.}
		}
}

\author{\small Xueliang Li, Yuan Si\\
\small Center for Combinatorics, Nankai University, Tianjin 300071, China\\
\small \tt lxl@nankai.edu.cn, yuan\_si@aliyun.com}
\date{}
\maketitle

\begin{abstract}
Let $\mathcal{E}$, $\mathcal{E}_1$, and $\mathcal{E}_2$ be equations, $n$ and $k$ be positive integers. The rainbow number $\operatorname{rb}([n],\mathcal{E})$ is difined as the minimum number of colors such that for every exact $(\operatorname{rb}([n],\mathcal{E}))$-coloring of $[n]$, there exists a rainbow solution of $\mathcal{E}$. The Gallai--Rado number $\operatorname{GR}_k(\mathcal{E}_1:\mathcal{E}_2)$ is defined as the minimum positive integer $N$, if it exists, such that for all $n\ge N$, every $k$-colored $[n]$ contains either a rainbow solution of $\mathcal{E}_1$ or a monochromatic solution of $\mathcal{E}_2$. In this paper, we get some exact values of rainbow and Gallai--Rado numbers involving binary function equations. We also provide an algorithm to calculate the rainbow numbers of nonlinear binary function equations.\\[3mm]
{\bf Keywords:} Coloring; Rainbow number; Gallai--Rado number; Binary function equation\\[3mm]
{\bf AMS subject classification 2020:} 05D10, 11B25, 11B75.
\end{abstract}

\section{Introduction}
Let $\mathbb{N}_{+}=\{1,2,\ldots\}$ be the set of positive natural numbers and $n$ be a positive integer. For convenience, we denote the set $\{1,2,\ldots,n\}$ as $[n]$. A \emph{$k$-coloring} of a set $[n]$ is a function $\chi: [n]\rightarrow [k]$, where $[k]$ is the set of colors. We usually use different numbers to represent different colors, and when the number of colors is relatively small, we can use specific color names, such as red, blue, green, and so on. If $\chi$ is a surjective, then we call the coloring $\chi$ \emph{exact}. In other words, an exact coloring requires that each color be used at least once. Obviously, if $\chi$ is exact, then the number of colors $k$ does not exceed the number of elements in the colored set $[n]$, that is, $k\le n$. All the colorings we consider in this paper are exact. In this paper, we use $y=f(x)$ to represent a binary function equation. If $x_0$, $y_0$ is a solution of $y=f(x)$, then we usually write the solution as $(x_0,y_0)$ or $(x_0,f(x_0))$.

\subsection{Schur and Rado numbers}

Ramsey theory is an important theory developed in the 20th century. In particular, Ramsey theory is widely used in graph theory and number theory, and is currently one of the hot and difficult research areas. In brief, Ramsey theory states that under sufficiently large structures, there must exist a substructure with specific properties. The main contribution is attributed to Ramsey, who published a pioneering paper \cite{Ramsey} in 1930. In fact, the study of Ramsey theory predates 1930. In 1916, Schur gave a theorem, later known as the Schur's theorem, which is one of the most important theorems in the early era of Ramsey theory. For a historical introduction to Ramsey theory, we refer to the first chapter of the monograph \cite{Soifer}.

\begin{theorem}{\upshape \cite{Schur}}\label{Schur's theorem}
For each integer $k\in \mathbb{N}_{+}$, there is a positive integer $S(k)$ such that every $k$-colored $[S(k)]$ contains a monochromatic solution of the equation $x+y=z$.
\end{theorem}

Generally speaking, if the colors of all numbers in a solution of an equation are the same, then we call the solution a \emph{monochromatic solution}; if the colors of all numbers in a solution of an equation are different, then we call the solution a \emph{rainbow solution}.

The $S(k)$ in Theorem~\ref{Schur's theorem} is called the \emph{Schur number}. To date, the exact values of the Schur numbers only for $1\le k\le 5$ have been determined. That is, $S(1)=2$, $S(2)=5$, $S(3)=14$, $S(4)=45$, and $S(5)=161$. For details, we refer to~\cite{AbbottHanson,Heule}.

Rado was one of the best Ph.D students of Schur. In 1933, Rado generalized Schur's theorem to general linear equations in his Ph.D thesis~\cite{Rado}. Rado called a linear equation $\sum_{i=1}^{n}a_ix_i=b$ \emph{$k$-regular} if there exists a monochromatic solution of the equation in any $k$-colored $\mathbb{N}_+$, where $k$ is a positive integer. If a linear equation $\sum_{i=1}^{n}a_ix_i=b$ is $k$-regular for all $k\in\mathbb{N}_+$, then the equation is \emph{regular}. According to the definition of regular equation and Theorem \ref{Schur's theorem}, the equation $x+y=z$ is regular. Rado gave a necessary and sufficient condition for judging whether a homogeneous linear equation is regular.

\begin{theorem}{\upshape \cite{Rado}}\label{Rado's theorem}
A homogeneous linear equation $\sum_{i=1}^{n}a_ix_i=0$ is regular if and only if there exists $I\subseteq [n]$ such that $\sum_{i\in I}a_i=0$. 
\end{theorem}

Inspired by Rado, the definition of regularity of equations can be extended beyond linear equations, as shown in the following definition.

\begin{definition}\label{Def-Regularity}
Let $k$ be a positive integer. An equation $\mathcal{E}$ is $k$-regular if there exists a monochromatic solution of the equation in any $k$-colored $\mathbb{N}_+$. If an equation $\mathcal{E}$ is $k$-regular for all $k\in\mathbb{N}_+$, then the equation is regular.
\end{definition}

As a generalization of Schur number, the following is the definition of Rado number. Noticing that, unlike Schur number, given a positive integer $k$ and an equation $\mathcal{E}$, the Rado number of $\mathcal{E}$ may not always exist.

\begin{definition}\label{Def-Rado number}
Let $\mathcal{E}$ be an equation and $k\ge 2$ be an integer. The Rado number $\operatorname{R}_k(\mathcal{E})$ is defined as the minimum positive integer, if it exists, such that every $k$-colored $[\operatorname{R}_k(\mathcal{E})]$ contains a monochromatic solution of $\mathcal{E}$.
\end{definition}

The following observation gives a basic method for proving the upper and lower bounds of the Rado number.

\begin{observation}\label{Obv-Rado number}
Let $\mathcal{E}$ be an equation and $k\ge 2$ be an integer. If there exists an exact $k$-coloring of $[n-1]$ such that there is no monochromatic solution of $\mathcal{E}$, then $\operatorname{R}_k(\mathcal{E})\ge n$. If every exact $k$-coloring of $[n]$ ensures that there is a monochromatic solution of $\mathcal{E}$, then $\operatorname{R}_k(\mathcal{E})\le n$. Moreover, if there exists an exact $k$-coloring of $\mathbb{N}_+$ such that there is no monochromatic solution of $\mathcal{E}$, then $\operatorname{R}_k(\mathcal{E})$ does not exist.
\end{observation}

Solving the exact value or upper and lower bounds of the Rado number of an equation has always been a hot topic. Some results on homogeneous linear equations were studied in references~\cite{HarborthMaasberg,HopkinsSchaal,JonesSchaal,RobertsonMyers,Saracino,SaracinoI,SaracinoII,SaracinoWynne,SchaalVestal,YangMaoHeWang}, some results on non-homogeneous linear equations were studied in references~\cite{BeutelspacherBrestovansky,DwivediTripathi,KosekSchaal,Myers,SaboSchaalTokaz,Schaal,SchaalZinter}, and some results on nonlinear equations were studied in references~\cite{Graham,Myers}. In fact, there are many other relevant references, which are not listed here.

\subsection{Rainbow and Gallai--Rado numbers}
In graph theory, a \emph{monochromatic subgraph} refers to a subgraph where the colors of all edges are the same, while a \emph{rainbow subgraph} refers to a subgraph where the colors of all edges are different. The definition of rainbow number in graph theory is relatively early, and we refer to the definition given by Schiermeyer~\cite{Schiermeyer} in 2004. However, the study of rainbow numbers in equations is later than that in graph theory, and we refer to the definition given by Fallon, Giles, Rehm, Wagner, and Warnberg \cite{FallonGilesRehmWagnerWarnberg} in 2020.

\begin{definition}{\upshape \cite{Schiermeyer}}\label{Def-rainbow number-graph}
	Let $G$ be a graph and $n$ be a positive integer. The rainbow number $\operatorname{rb}(n,G)$ is difined as the minimum number of colors such that for every exact $(\operatorname{rb}(n,G))$-coloring of complete graph $K_n$, there exists a rainbow subgraph $G$.
\end{definition}

\begin{definition}{\upshape \cite{FallonGilesRehmWagnerWarnberg}}\label{Def-rainbow number-equation}
	Let $\mathcal{E}$ be an equation and $n$ be a positive integer. The rainbow number $\operatorname{rb}([n],\mathcal{E})$ is difined as the minimum number of colors such that for every exact $(\operatorname{rb}([n],\mathcal{E}))$-coloring of $[n]$, there exists a rainbow solution of $\mathcal{E}$.
\end{definition}

Fallon, Giles, Rehm, Wagner, and Warnberg also gave the following observation, which is the basic method for proving the upper and lower bounds of the rainbow number.

\begin{observation}{\upshape \cite{FallonGilesRehmWagnerWarnberg}}\label{Obv-rainbow number}
	If there exists an exact $(k-1)$-coloring of $[n]$ such that the equation $\mathcal{E}$ has no rainbow solution in $[n]$, then $\operatorname{rb}([n],\mathcal{E})\ge k$. If every exact $k$-coloring of $[n]$ ensures that the equation $\mathcal{E}$ has a rainbow solution in $[n]$, then $\operatorname{rb}([n],\mathcal{E})\le k$.
\end{observation}

In 1967, Gallai's paper~\cite{Gallai} for the first time revealed the structure of colored complete graphs without rainbow triangles, and thus gave rise to a new research direction in graph theory known as the \emph{Gallai--Ramsey number}. In 2010, Faudree, Gould, Jacobson, and Magnant in~\cite{FGJM10} defined Gallai--Ramsey number $\operatorname{gr}_k(G:H)$. For more information about Gallai--Ramsey number, we refer to the monograph \cite{MagnantNowbandegani}.

\begin{definition}{\upshape \cite{FGJM10}}\label{Def-Gallai-Ramsey number}
	Given two non-empty graphs $G,H$ and a positive integer $k$, define the Gallai--Ramsey number $\operatorname{gr}_k(G:H)$
	to be the minimum integer $N$ such that for all $n\ge N$, every $k$-edge-colored complete graph $K_n$ contains either a rainbow subgraph $G$ or a monochromatic subgraph $H$.
\end{definition}

Due to the rapid development of the Gallai--Ramsey number in the past decade, it has become a hot research area in graph theory. Inspired by this problem, the definition of the \emph{Gallai--Schur number} was proposed in Budden's paper~\cite{Budden} in 2020, introducing this type of problem from graph theory to number theory. Since the Gallai--Schur number only studies the equation $x+y=z$, it can be generalized to other equations, and similarly defined as the \emph{Gallai--Rado number}. The Gallai--Rado numbers were first studied by Mao, Robertson, Wang, Yang, and Yang~\cite{MaoRobertsonWangYangYang}. It can be said that the Gallai--Rado number is a generalized definition of the Gallai--Schur number.

\begin{definition}\label{Def-Gallai-Rado number}
Let $\mathcal{E}_1$ and $\mathcal{E}_2$ be two equations and $k$ be a positive integer. The Gallai--Rado number $\operatorname{GR}_k(\mathcal{E}_1:\mathcal{E}_2)$ is defined as the minimum positive integer $N$, if it exists, such that for all $n\ge N$, every $k$-colored $[n]$ contains either a rainbow solution of $\mathcal{E}_1$ or a monochromatic solution of $\mathcal{E}_2$.
\end{definition}

For simplicity, without causing confusion, we sometimes say that $\mathcal{E}_1$ in $\operatorname{GR}_k(\mathcal{E}_1:\mathcal{E}_2)$ is the \emph{rainbow equation} and $\mathcal{E}_2$ is the \emph{monochromatic equation}. Noticing that one of the biggest differences between the Gallai--Rado number and the Gallai--Ramsey number is that the Gallai--Rado number does not always exist. Next, we present an observation on the basic method for proving the upper and lower bounds of the Gallai--Rado number.

\begin{observation}\label{Obv-Gallai-Rado number}
	If there exists an exact $k$-coloring of $[N-1]$ such that there is neither a rainbow solution of $\mathcal{E}_1$ nor a monochromatic solution of $\mathcal{E}_2$, then $\operatorname{GR}_k(\mathcal{E}_1:\mathcal{E}_2)\ge N$. If every exact $k$-coloring of $[n] \,(n\ge N)$ ensures that there is either a rainbow solution of $\mathcal{E}_1$ or a monochromatic solution of $\mathcal{E}_2$, then $\operatorname{GR}_k(\mathcal{E}_1:\mathcal{E}_2)\le N$. Moreover, if there exists an exact $k$-coloring of $\mathbb{N}_+$ such that there is neither a rainbow solution of $\mathcal{E}_1$ nor a monochromatic solution of $\mathcal{E}_2$, then $\operatorname{GR}_k(\mathcal{E}_1:\mathcal{E}_2)$ does not exist.
\end{observation}

\subsection{Main results}
In Section $2$, we introduce some new definitions that are important in Sections $3$ and $4$. For example, the $\lambda$-class defined in Section $2$ is used to describe the colored structure of $[n]$ without rainbow solution of the equation $y=ax+b$.

In Section $3$, we first give the exact value of the rainbow number of the binary linear equation $y=ax+b$. Then, for nonlinear binary function equations, the formulas for their rainbow numbers are not easily given directly. Therefore, we define a parameter called monochromatic parameter. Due to the close relationship between the monochromatic parameter we defined and the rainbow number, we also give an algorithm to solve the monochromatic parameters of nonlinear binary function equations.

In Section $4$, we first give the general properties and related connections of the Rado numbers and Gallai--Rado numbers of the binary linear equation $y=ax+b$, and then we present one of the main results, which gives the exact values of the Gallai--Rado numbers of the rainbow equation $y=x+b$ versus monochromatic general multivariate linear equations with the fixed number of colors $b$. Next, we keep the rainbow equation as $y=x+b$, consider the monochromatic nonlinear binary function equations, and give the exact values of the Gallai--Rado numbers.

In Section $5$, as a conclusion, we give the relationship between the rainbow numbers and the Gallai--Rado numbers of equations, and explain why our study of the rainbow equation for the Gallai--Rado numbers in this paper is only $y=x+b$. Finally, we raise some questions and give some prospects.

\section{Preliminaries}
In the following definition, the $\lambda$-class $\mathcal{C}_{(y=f(x),\lambda)}$ we provide is crucial for some of our results.

\begin{definition}\label{Def-lambda class}
	Let $y=f(x)$ be a strictly monotonically increasing binary function equation such that $f(x)\in \mathbb{N}_+$ for all $x\in\mathbb{N}_+$, and let integer $n\ge f(1)$. For each $\lambda\in [n]$, we define the $\lambda$-class of $y=f(x)$ as $\mathcal{C}_{(y=f(x),\lambda)} =\left\{\lambda,f(\lambda),f(f(\lambda)),\ldots\right\}\subseteq [n]$. If all the numbers in $\mathcal{C}_{(y=f(x),\lambda)}$ are of the same color, then we call the set $\mathcal{C}_{(y=f(x),\lambda)}$ monochromatic, and the color of the set $\mathcal{C}_{(y=f(x),\lambda)}$ is the same as the color of the numbers in it. Moreover, the $\lambda$-class can also be defined as a subset of $\mathbb{N}_{+}$, that is, $\mathcal{C}_{(y=f(x),\lambda)} =\left\{\lambda,f(\lambda),f(f(\lambda)),\ldots\right\}\subseteq \mathbb{N}_{+}$ for each $\lambda\in\mathbb{N}_{+}$.
\end{definition}

\begin{lemma}\label{Lem-General term of a sequence}
Let ingeters $a\ge 1$ and $b\ge 0$. If the recurrence of the sequence $\{x_i: i\ge 1\}$ is $x_{i+1}=ax_{i}+b$, then the general term is $x_{i+1}=a^{i+1}\left(\sum_{j=1}^{i}\frac{b}{a^{j+1}}+\frac{x_1}{a}\right)$.
\end{lemma}
\begin{proof}
	The recurrence of $\{x_i: i\ge 1\}$ can be rewritten as
	\begin{equation*}
		\begin{split}
			\frac{x_{i+1}}{a^{i+1}} &= \frac{b}{a^{i+1}}+\frac{x_{i}}{a^{i}} \\
			&= \frac{b}{a^{i+1}}+\frac{b}{a^{i}}+\frac{x_{i-1}}{a^{i-1}} \\
			& \qquad\qquad\vdots \\
			&= \sum_{j=1}^{i}\frac{b}{a^{j+1}}+\frac{x_{1}}{a}.
		\end{split}
	\end{equation*}
The result thus follows.
\end{proof}

Based on Lemma \ref{Lem-General term of a sequence} and Definition \ref{Def-lambda class}, we give an example of the $\lambda$-class of $y=ax+b$.

\begin{example}\label{Ex-lambda class y=ax+b}
For integers $a\ge 1$, $b\ge 0$, and $n\ge a+b$, and for each $\lambda\in [n]$, we have
	\begin{equation*}
		\mathcal{C}_{(y=ax+b,\lambda)} =\left\{a^{i}\left(\sum_{j=1}^{i-1}\frac{b}{a^{j+1}}+\frac{\lambda}{a}\right): i\in\mathbb{N}_{+}~\text{and}~ \sum_{j=1}^{0}\frac{b}{a^{j+1}}\overset{\textbf{def}}{=}0\right\}\subseteq [n].
	\end{equation*} 
\end{example}

Next, we present the colored structure theorem for $[n]$ without rainbow solution of $y=f(x)$.

\begin{theorem}\label{Thm-Colored structure}
Let $y=f(x)$ be a strictly monotonically increasing binary function equation such that $f(x)\in \mathbb{N}_+$ for all $x\in\mathbb{N}_+$, and let integer $n\ge f(1)$. A colored set $[n]$ contains no rainbow solution of $y=f(x)$ if and only if for each $\lambda\in [n]$, the $\lambda$-class $\mathcal{C}_{(y=f(x),\lambda)}\subseteq [n]$ is monochromatic.
\end{theorem}
\begin{proof}
Firstly, we assume that for each $\lambda\in [n]$, the $\lambda$-class $\mathcal{C}_{(y=f(x),\lambda)}\subseteq [n]$ is monochromatic. Let $(x_0,y_0)$ be an arbitrary solution of the binary function equation $y=f(x)$ in $[n]$. Since $x_0,y_0 \in \mathcal{C}_{(y=f(x),x_0)}=\{x_0,f(x_0),f(f(x_0)),\ldots\}\subseteq [n]$, it follows that $[n]$ contains no rainbow solution of $y=f(x)$.
	
Next, we assume that the colored set $[n]$ contains no rainbow solution of $y=f(x)$. Thus, for each $\lambda\in [n]$, $\lambda$ and $f(\lambda)$ are of the same color. Similarly, $f(\lambda)$ and $f(f(\lambda))$ must also be of the same color. According to the recursion, we get all the numbers in $\left\{\lambda, f(\lambda), f(f(\lambda)), \ldots \right\}\subseteq [n]$ are of the same color, which implies that the $\lambda$-class $\mathcal{C}_{(y=f(x),\lambda)}\subseteq [n]$ is monochromatic. The result thus follows.
\end{proof}

In order to maximize the number of colors used without the rainbow solution of $y=f(x)$ in the colored set $[n]$, based on Theorem \ref{Thm-Colored structure}, we construct a coloring of $[n]$ as follows.

\begin{definition}\label{Def-Colored structure}
Let $\chi$ be a coloring of $[n]$ as follows: For each $\lambda\in [n]$, the $\lambda$-class $\mathcal{C}_{(y=f(x),\lambda)}\subseteq [n]$ is monochromatic and for each pair of different $\lambda_i$ and $\lambda_j$ satisfies $\mathcal{C}_{(y=f(x),\lambda_i)}\cap\mathcal{C}_{(y=f(x),\lambda_j)}=\emptyset$, $\mathcal{C}_{(y=f(x),\lambda_i)}$ and $\mathcal{C}_{(y=f(x),\lambda_j)}$ are of different colors.
\end{definition}

According to Theorem~\ref{Thm-Colored structure}, the coloring $\chi$ of $[n]$ given in Definition $\ref{Def-Colored structure}$ is the coloring that maximize number of colors used and ensure that $y=f(x)$ has no rainbow solution.

\begin{corollary}\label{Cor-Colored structure}
Let ingeters $b\ge 2$ and $n\ge b+1$. A $b$-colored set $[n]$ contains no rainbow solution of $y=x+b$ if and only if for each $\lambda\in [b]$, the $\lambda$-class $\mathcal{C}_{(y=x+b,\lambda)}\subseteq [n]$ is monochromatic. Furthermore, for different $\lambda_i$ and $\lambda_j$ in $[b]$, $\mathcal{C}_{(y=x+b,\lambda_i)}$ and $\mathcal{C}_{(y=x+b,\lambda_j)}$ have different colors.
\end{corollary}

\section{Results for rainbow numbers}
\begin{theorem}\label{Thm-Rainbow number y=ax+b}
	For ingeters $a\ge 1$, $b\ge 0$, and $n\ge a+b$, we have
	\begin{equation*}
		\operatorname{rb}([n],y=ax+b)=n-\left\lfloor\frac{n-b}{a}\right\rfloor+1.
	\end{equation*}
\end{theorem}
\begin{proof}
For the lower bound, we only need to construct a $\left(n-\left\lfloor\frac{n-b}{a}\right\rfloor\right)$-coloring of $[n]$, so that the equation $y=ax+b$ has no rainbow solution in $[n]$. Actually, the coloring $\chi$ of $[n]$ we constructed is follows from Definition~\ref{Def-Colored structure}. The specific steps for constructing coloring are given in Algorithm~\ref{Algorithm 1}.
	
\begin{algorithm}[H]
\caption{The coloring $\chi$ of $[n]$.}\label{Algorithm 1}
\begin{algorithmic}[1]
\REQUIRE Ingeters $a\ge 1$, $b\ge 0$, and $n\ge a+b$.

\ENSURE A colored set $[n]$ without rainbow solution of $y=ax+b$.

\STATE Let $t=1$ and $S=\emptyset$.

\WHILE{$t\le n$}
	
  \IF {$t\notin \mathcal{C}_{(y=ax+b,\lambda)}$ for all $\lambda\in S$} 

    \STATE assign a new color to $t$
    
    \STATE $S=S\cup\{t\}$, $t=t+1$
    
  \ELSE
      
    \STATE the color assigned to $t$ is the same as the color of $\mathcal{C}_{(y=ax+b,\lambda_{0})}$, where $t\in \mathcal{C}_{(y=ax+b,\lambda_{0})}$  
    
    \STATE $t=t+1$
    
  \ENDIF

\ENDWHILE   
\end{algorithmic}
\end{algorithm}
	
Next, we calculate how many colors are used for coloring $\chi$. According to recursion, for each $\lambda\in [n]$, each number in $\mathcal{C}_{(y=ax+b,\lambda)}\setminus\{\lambda\}$ can be written as $ia+b \, (i\in\mathbb{N}_+)$, so all numbers in the subset $\{ia+b: i\in\mathbb{N}_+\}\subseteq [n]$ cannot be assigned new colors. Therefore, we only need to calculate the number of numbers in $[n]$ that cannot be assigned new colors. For the convenience of counting, we list the following table. Numbers with lightgray boxes in the table cannot be assigned new colors, while other numbers are assigned different colors. It is easy to see that there are $\left\lfloor\frac{n-b}{a}\right\rfloor$ numbers that cannot be assigned new colors, that is, the coloring $\chi$ uses $n-\left\lfloor\frac{n-b}{a}\right\rfloor$ colors.
	
\begin{center}
\begin{tabular}{|c|c|c|c|c|c|c|}
\hline
\multicolumn{7}{|c|}{Count the number of colors used for coloring $\chi$} \\[0.1cm]
\hline 
$1$ & \ldots & $b$ & $b+1$ & \ldots & $a+b-1$ & \cellcolor{lightgray}$a+b$   \\ \hline 
\diagbox[height=1.5em] & \ldots & \diagbox[height=1.5em] & $a+b+1$ & \ldots & $2a+b-1$ & \cellcolor{lightgray}$2a+b$ \\ \hline 
\vdots & \vdots & \vdots & \vdots & \vdots & \vdots & \cellcolor{lightgray}\vdots \\ \hline
\diagbox[height=1.5em] & \ldots & \diagbox[height=1.5em] & $(i-1)a+b+1$ & \ldots & $ia+b-1$ & \cellcolor{lightgray}$ia+b$ \\ \hline 
\vdots & \vdots & \vdots & \vdots & \vdots & \vdots & \cellcolor{lightgray}\vdots \\ \hline
\diagbox[height=1.5em] & \ldots & \diagbox[height=1.5em] & \ldots & $n$ & \ldots & \cellcolor{lightgray}\ldots \\ \hline
\end{tabular}
\end{center}
	
For the upper bound, we arbitrarily color $[n]$ with $n-\left\lfloor\frac{n-b}{a}\right\rfloor+1$ colors. Recall that if there is no rainbow solution of $y=ax+b$ in $[n]$, then according to Theorem \ref{Thm-Colored structure}, each $\lambda$-class is monochromatic. Noticing that under the coloring $\chi$ constructed above, all the numbers in $\mathcal{C}_{(y=ax+b,\lambda)}$ have the same color for each $\lambda\in [n]$, and if each pair of different $\lambda_i$ and $\lambda_j$ satisfies $\mathcal{C}_{(y=f(x),\lambda_i)}\cap\mathcal{C}_{(y=f(x),\lambda_j)}=\emptyset$, $\mathcal{C}_{(y=ax+b,\lambda_i)}$ and $\mathcal{C}_{(y=ax+b,\lambda_j)}$ are of different colors. We have counted $n-\left\lfloor\frac{n-b}{a}\right\rfloor$ colors used for the coloring $\chi$, and if an additional color is added, then according to pigeonhole principle, it will inevitably lead to the existence of an integer $\lambda_0\in [n]$, resulting in two numbers with different colors in $\mathcal{C}_{(y=ax+b,\lambda_0)}$, which implies that there is a rainbow solution of $y=ax+b$ in $[n]$. The result thus follows.
\end{proof}

We consider the rainbow number of binary function equations. Assuming that strictly monotonically increasing binary function equations satisfy $f(x)\in \mathbb{N}_+$ for all $x\in\mathbb{N}_+$. Thus, there is an inverse function $x=f^{-1}(y)$ for $y=f(x)$, and if $y$ is not a positive integer, then $f^{-1}(y)$ must also be not a positive integer. If $y=f(x)$ has no solution in $[n]$, then naturally there is no rainbow solution. Therefore, it is meaningful to study the rainbow number of $y=f(x)$ only when $y=f(x)$ has at least one solution in $[n]$, that is, $n\ge f(1)$. Now we know that $y=f(x)$ has a solution in $[n]$, and the natural idea is that if all numbers in $[n]$ are assigned different colors, then there must be a rainbow solution of $y=f(x)$, but in this case, the number of colors is the largest. Next, we try to color some of the numbers in $[n]$ the same color in an attempt to reach the extreme value. Definition $\ref{Def-monochromatic parameter}$ is to define this extreme value characteristic.

\begin{definition}\label{Def-monochromatic parameter}
Let $y=f(x)$ be a strictly monotonically increasing binary function equation such that $f(x)\in \mathbb{N}_+$ for all $x\in\mathbb{N}_+$, and integers $n\ge f(1)$ and $\mu\ge 1$. If there exists a $(n-\mu)$-coloring of $[n]$ such that $y=f(x)$ has no rainbow solution in $[n]$, but for all $(n-\mu+1)$-coloring of $[n]$, $y=f(x)$ always has a rainbow solution in $[n]$, then we call $\mu$ the monochromatic parameter with respect to $y=f(x)$ and $n$.
\end{definition}

Based on Theorem~\ref{Thm-Rainbow number y=ax+b} and Definition~\ref{Def-monochromatic parameter}, we give an example of the monochromatic parameter $\mu$ with respect to $y=ax+b$ and $n$.

\begin{example}\label{Ex-monochromatic parameter}
Let integers $a\ge 1$, $b\ge 0$, and $n\ge a+b$. The monochromatic parameter with respect to $y=ax+b$ and $n$ is
	\begin{equation*}
		\mu= \left\lfloor\frac{n-b}{a}\right\rfloor.
	\end{equation*} 
\end{example}

Combining Definition~\ref{Def-monochromatic parameter} and Observation~\ref{Obv-rainbow number}, we directly provide the following observation.

\begin{observation}
Let $y=f(x)$ be a strictly monotonically increasing binary function equation such that $f(x)\in \mathbb{N}_+$ for all $x\in\mathbb{N}_+$, and integer $n\ge f(1)$. If $\mu$ is the monochromatic parameter with respect to $y=f(x)$ and $n$, then
	\begin{equation*}
		\operatorname{rb}([n],y=f(x))=n-\mu+1.
	\end{equation*} 
\end{observation}

We can obtain the explicit expression for the monochromatic parameter $\mu$ with respect to $y=ax+b$ and $n$ due to the linear properties of $y=ax+b$. However, for a nonlinear binary function equation $y=f(x)$ that satisfies strictly monotonically increasing and $f(x)\in \mathbb{N}_+$ for all $x\in\mathbb{N}_+$, it is not easy to directly obtain the explicit expression of its monochromatic parameter $\mu$. The Algorithm~\ref{Algorithm 2} we provide next can be applied to find the monochromatic parameters of the general nonlinear binary function equations.

\begin{algorithm}[H]
\caption{Calculate the monochromatic parameter $\mu$.}\label{Algorithm 2}
\begin{algorithmic}[1]
\REQUIRE A strictly monotonically increasing binary function equation $y=f(x)$ such that $f(1)\ge 2$ and $f(x)\in \mathbb{N}_+$ for all $x\in\mathbb{N}_+$, and integer $n\ge f(1)$.

\ENSURE The monochromatic parameter $\mu$ with respect to $y=f(x)$ and $n$.

\STATE Let $\mu=0, t=\left\lfloor f^{-1}(n) \right\rfloor$, and $S=\emptyset$.

\WHILE{$t\ge 1$}
	
  \IF {$t\notin S$} 

    \STATE $\mu=\mu+1$
    
      \IF {$f^{-1}(t)\ge 1$ is an integer} 
      
        \STATE $s=t$
          
          \WHILE{$f^{-1}(s)\ge 1$ is an integer}
          
            \STATE $S=S\cup \left\{s, f^{-1}(s)\right\}, \mu=\mu+1, s=f^{-1}(s)$
            
          \ENDWHILE
          
        \STATE $t=t-1$
        
      \ELSE
      
    \STATE $S=S\cup \{t\},t=t-1$    
          
  \ENDIF
  
\ELSE

  \STATE $t=t-1$

\ENDIF
\ENDWHILE   
\RETURN $\mu$.
\end{algorithmic}
\end{algorithm}

\begin{remark}
In order for Algorithm $\ref{Algorithm 2}$ to run properly, the While-Do loop from step $7$ to step $9$ cannot be a dead loop, that is, there cannot be a $x_0\in \mathbb{N}_+$ such that $f(x_0)=x_0$. In fact, we do not need to worry about this problem. Since $y=f(x)$ is strictly monotonically increasing, $f(1)\ge 2$ and $f(x)\in \mathbb{N}_+$ for all $x\in\mathbb{N}_+$, it follows that $f(2)\ge 3$. Otherwise, this contradicts the strictly monotonically increasing of $y=f(x)$. According to recursion, for all $x\in \mathbb{N}_+$, $f(x)\ge x+1$. Therefore, there is no such $x_0\in \mathbb{N}_+$ that $f(x_0)=x_0$.
	
In addition, for binary function equations such that $f(1)=1$, $f(2)\ge 3$, and strictly monotonically increasing (for example, $y=x^a$, where integer $a\ge 2$), we only need to modify some of the conditions for Algorithm~\ref{Algorithm 2}, that is, to replace $t\ge 1$, $f^{-1}(t)\ge 1$, and $f^{-1}(s)\ge 1$ in steps $2$, $5$, and $7$ with $t\ge 2$, $f^{-1}(t)\ge 2$, and $f^{-1}(s)\ge 2$, respectively.
\end{remark}

\begin{example}
We implement Algorithm~\ref{Algorithm 2} in Python to solve $\operatorname{rb}([n],y=ax+b)$ for ingeters $a\ge 1$ and $b\ge 0$. The computational results of the computer are completely consistent with the exact values given in Theorem~\ref{Thm-Rainbow number y=ax+b}. Here is the complete Python code.
\end{example}
\begin{verbatim}
import math
a=int(input("Solving the rainbow number of y=ax+b in [n] \na="))
b=int(input("b="))
n=int(input("n="))
if n<a+b:
    print("Please re-enter n")
else:
    def f(y):
        return (y-b)/a 
    mu=0
    t=math.floor(f(n)) 
    S={""}
    while t>=1:
        if t not in S:
            mu=mu+1
            if f(t)>=1 and isinstance(f(t),int):
                   s=t
                   while f(s)>=1 and isinstance(f(s),int):
                    S=S.union({s,f(s)})
                    mu=mu+1
                    s=f(s)
            else: 
                  S=S.union({t})
                  t=t-1
        else:
            t=t-1
rb=n-mu+1  
print(f"The rainbow number of y={a}x+{b} in [{n}] is {rb}")
\end{verbatim}

\section{Results for Gallai--Rado numbers}
We first prove that the Rado number for equation $y=ax+b$ does not exist, except for $y=x$.

\begin{lemma}\label{Lem 1}
	Let integers $k\ge r\ge 2$. If $\mathcal{E}$ is not $r$-regular, then $\operatorname{R}_k(\mathcal{E})$ does not exist.
\end{lemma}
\begin{proof}
Since the equation $\mathcal{E}$ is not $r$-regular, there exists an $r$-coloring $\chi$ of $\mathbb{N}_+$, so that the equation $\mathcal{E}$ has no monochromatic solution in $\mathbb{N}_+$. For each integer $k\ge t$, we assign new colors to some numbers of the same color based on the $r$-coloring $\chi$ of $\mathbb{N}_+$, thereby constructing a $k$-coloring $\chi'$ of $\mathbb{N}_+$. Noticing that $\mathbb{N}_+$ is an infinite set, while $k$ is a finite number of colors, so we can always add new colors to construct a $k$-coloring $\chi'$ of $\mathbb{N}_+$, which implies that the equation $\mathcal{E}$ is not $k$-regular. It follows from Observation \ref{Obv-Rado number} that $\operatorname{R}_k(\mathcal{E})$ does not exist.
\end{proof}

\begin{theorem}
For integers $k\ge 2$, $a\ge 1$, $b\ge 0$, and $(a,b)\ne (1,0)$, $\operatorname{R}_k(y=ax+b)$ does not exist.
\end{theorem}
\begin{proof}
From Lemma~\ref{Lem 1}, it is sufficient to show that $y=ax+b$ is not $2$-regular. The red/blue-coloring of $\mathbb{N}_+$ constructed is as follows: For integer $i\ge 1$, we consider the following set and agree that $\sum_{j=1}^{0}\frac{b}{a^{j+1}}=0$. 
	\begin{equation*}
		\mathcal{A}_i=\left\{a^i\left(\sum_{j=1}^{i-1}\frac{b}{a^{j+1}}+\frac{1}{a}\right),a^i\left(\sum_{j=1}^{i-1}\frac{b}{a^{j+1}}+\frac{1}{a}\right)+1,\ldots, a^{i+1}\left(\sum_{j=1}^{i}\frac{b}{a^{j+1}}+\frac{1}{a}\right)-1\right\}.
	\end{equation*}	
As can be seen, $\bigcup_{i=1}^{\infty}\mathcal{A}_{i}=\mathbb{N}_{+}$. When the positive integer $i$ is odd, we color all the numbers in set $\mathcal{A}_i$ red; and when the positive integer $i$ is even, we color all the numbers in set $\mathcal{A}_i$ blue. It is easy to verify that under this coloring, the equation $y=ax+b$ does not have a monochromatic solution in $\mathbb{N}_{+}$. The result thus follows.
\end{proof}

Obviously, with exact $k$-coloring, $\operatorname{R}_k(y=x)=k$ and there never be any rainbow solution of $y=x$, so the following theorem is straightforward.
\begin{theorem}
	For ingeter $k\ge 2$ and an arbitrary equation $\mathcal{E}$, we have
	\begin{equation*}
		\operatorname{GR}_{k}\left(y=x:\mathcal{E}\right)=\operatorname{R}_k(\mathcal{E}).
	\end{equation*}
\end{theorem}

Let integers $n\ge k\ge b+1$ and $b\ge 1$. According to Corollary~\ref{Cor-Colored structure}, when $k\ge b+1$, any $k$-coloring of $[n]$, the equation $y=x+b$ must have a rainbow solution in $[n]$. The following theorem is directly given.
\begin{theorem}
	For ingeters $k\ge b+1$, $b\ge 1$, and an arbitrary equation $\mathcal{E}$, we have
	\begin{equation*}
		\operatorname{GR}_{k}\left(y=x+b:\mathcal{E}\right)=k.
	\end{equation*}
\end{theorem}

Next, we provide the result of Gallai--Rado numbers involving monochromatic general linear equations.
\begin{theorem}
	For ingeters $a_i\ge 1$ for all $i\in[t]$, $t\ge 1$, $b\ge 2$, and $c\ge 0$, we have
	\begin{equation*}
		\operatorname{GR}_{b}\left(y=x+b:y=\sum_{i=1}^{t}a_ix_i+c\right)=\sum_{i=1}^{t}\left(\lambda_{\min}+(i-1)b\right)a_i+c,
	\end{equation*}
	where $\lambda_{\min}=\min\left\{\lambda : \text{ $\lambda\in [b]$ and $\left(\sum_{i=1}^{t}a_i\lambda+c-\lambda\right) \equiv 0 \pmod{b}$} \right\}$. Moreover, if $\lambda_{\min}$ does not exist, then $\operatorname{GR}_{b}\left(y=x+b:y=\sum_{i=1}^{t}a_ix_i+c\right)$ also does not exist.
\end{theorem}
\begin{proof}
Let $N=\sum_{i=1}^{t}\left(\lambda_{\min}+(i-1)b\right)a_i+c$ and $\mathcal{C}_{(y=x+b,\lambda)} =\left\{\lambda+(i-1)b: i\in\mathbb{N}_{+}\right\}\subseteq [N-1]$, where $\lambda\in [b]$ and $\lambda_{\min}=\min\left\{\lambda : \text{ $\lambda\in [b]$ and $\left(\sum_{i=1}^{t}a_i\lambda+c-\lambda\right) \equiv 0 \pmod{b}$} \right\}$. For the lower bound, we only need to construct a $b$-colored $[N-1]$ such that there is neither a rainbow solution of $y=x+b$ nor a monochromatic solution of $y=\sum_{i=1}^{t}a_ix_i+c$ in $[N-1]$. Since there is no rainbow solution of $y=x+b$ in $[N-1]$, it follows from Corollary~\ref{Cor-Colored structure} that for each $\lambda\in [b]$, the set $\mathcal{C}_{(y=x+b,\lambda)}$ is monochromatic. Therefore, the colored structure of $[N-1]$ is uniquely determined without considering the order of colors. Next, we only need to prove that $[N-1]$ does not contain a monochromatic solution of $y=\sum_{i=1}^{t}a_ix_i+c$. 
	
To the contrary, if the equation $y=\sum_{i=1}^{t}a_ix_i+c$ has a monochromatic solution in $[N-1]$, then there exists $\lambda_{0}\in [b]$ so that there are integers $x'_1,x'_2,\ldots,x'_t$ and $y'=\sum_{i=1}^{t}a_ix'_i+c$ are all in $\mathcal{C}_{(y=x+b,\lambda_0)}$. Since $x'_1,x'_2,\ldots,x'_t$ in $\mathcal{C}_{(y=x+b,\lambda_0)}$, it follows that there are positive ingeters $j_1,j_2,\ldots,j_t$ such that 
	\begin{equation*}
		\begin{split}
			&	x'_1=\lambda_0+j_1b, \\
			&	x'_2=\lambda_0+j_2b, \\
			&	\qquad\vdots \\
			&	x'_t=\lambda_0+j_tb.
		\end{split}
	\end{equation*}
	Therefore, 
	\begin{equation*}
		\begin{split}
			y' &= \sum_{i=1}^{t}a_ix'_i+c \\
			&= \sum_{i=1}^{t}a_i(\lambda_0+j_ib)+c \\
			&= \sum_{i=1}^{t}a_i\lambda_0+\left(\sum_{i=1}^{t}a_ij_i\right)\cdot b+c \\
			&= \lambda_{0}+\left(\sum_{i=1}^{t}a_ij_i\right)\cdot b+\sum_{i=1}^{t}a_i\lambda_0+c-\lambda_{0}
		\end{split}		
	\end{equation*}
in $\mathcal{C}_{(y=x+b,\lambda_0)}$ if and only if $\left(\sum_{i=1}^{t}a_i\lambda_0+c-\lambda_0\right) \equiv 0 \pmod{b}$. We can see that if such $\lambda_{0}$ does not exist, then the equation $y=\sum_{i=1}^{t}a_ix_i+c$ cannot have a monochromatic solution in $\mathbb{N}_+$, which implies that $\operatorname{GR}_{b}\left(y=x+b:y=\sum_{i=1}^{t}a_ix_i+c\right)$ does not exist. If such $\lambda_{0}$ exists, then we choose the smallest one, $\lambda_{\min}$, to obtain
	\begin{equation*}
		\begin{split}
			&y' = \sum_{i=1}^{t}a_ix'_i+c \\
			&\ge \lambda_{\min}a_1+\left(\lambda_{\min}+b\right)a_2+\cdots+\left(\lambda_{\min}+(t-1)b\right)a_t+c \\
			&= \sum_{i=1}^{t}\left(\lambda_{\min}+(i-1)b\right)a_i+c=N,
		\end{split}		
	\end{equation*}
which contradicts with $[N-1]$.
	
For the upper bound, we consider any $b$-colored $[n]\,(n\ge N)$. If the equation $y=x+b$ has a rainbow solution in $[n]$, then it is done. Otherwise, the equation $y=x+b$ does not have a rainbow solution in $[n]$, and in this case, the colored structure of $[n]$ is uniquely determined. This means that for each $\lambda\in [b]$, the set $\mathcal{C}_{(y=x+b,\lambda)}=\left\{\lambda+(i-1)b: i\in\mathbb{N}_{+}\right\}\subseteq [n]$ is monochromatic. As discussed above, if the defined $\lambda_{\min}$ does not exist, then $\operatorname{GR}_{b}\left(y=x+b:y=\sum_{i=1}^{t}a_ix_i+c\right)$ also does not exist. Therefore, we assume that $\lambda_{\min}$ exists. In this case, the equation $y=\sum_{i=1}^{t}a_ix_i+c$ has a monochromatic solution in $\mathcal{C}_{(y=x+b,\lambda_{\min})}$, where one of the monochromatic solution is 
	\begin{equation*}
		\begin{split}
			&	x_1=\lambda_{\min}, \\
			&	x_2=\lambda_{\min}+b, \\
			&	\qquad\vdots \\
			&	x_t=\lambda_{\min}+(t-1)b,\\
			&  y=\sum_{i=1}^{t}a_ix_i+c=\sum_{i=1}^{t}\left(\lambda_{\min}+(i-1)b\right)a_i+c.
		\end{split}
	\end{equation*}
The result thus follows.
\end{proof}

Regarding monochromatic nonlinear binary function equations, we have the following results.

\begin{theorem}
	Let integers $a\ge 1$, $b\ge 0$, and $c\ge 2$. Then
	\begin{equation*}
		\operatorname{GR}_{2}\left(y=x+2:y=ax^c+b\right)\left\{
		\begin{array}{ll}
			\text{does not exist}, &  \text{if $a$ and $b$ are odd;} \\
			= a+b,  &  \text{if $a$ and $b$ have different parity;} \\
			= a\cdot 2^c+b, &  \text{if $a$ and $b$ are even.} \\
		\end{array} \right.
	\end{equation*}
\end{theorem}
\begin{proof}
We distinguish the following three cases to show this theorem.	
\setcounter{case}{0}
\begin{case}
$a$ and $b$ are odd.
\end{case}
According to Observation~\ref{Obv-Gallai-Rado number}, we only need to construct a red/blue-coloring of $\mathbb{N}_{+}$ such that there is neither a rainbow solution of $y=x+2$ nor a monochromatic solution of $y=ax^c+b$. In fact, we only need to color all odd numbers red and all even numbers blue in $\mathbb{N}_{+}$. Under this coloring, there is clearly no rainbow solution of $y=x+2$. Let $(x_0, y_0)$ be any positive integer solution of $y=ax^c+b$. If $x_0$ is even, then $x_{0}^c$ is also even, indicating that $ax_{0}^c$ is also even. Since $b$ is odd, it follows that $y_0=ax_{0}^c+b$ is odd, which implies that the colors of $x_0$ and $y_0$ are different. If $x_0$ is odd, then $x_{0}^c$ is also odd, indicating that $ax_{0}^c$ is also odd. Since $b$ is odd, it follows that $y_0=ax_{0}^c+b$ is even, which implies that the colors of $x_0$ and $y_0$ are different. Therefore, there is no monochromatic solution of $y=ax^c+b$ in such red/blue-colored $\mathbb{N}_{+}$. The result thus follows.
	
\begin{case}
$a$ and $b$ have different parity.
\end{case}
For the lower bound, we color all odd numbers red and all even numbers blue in $[a+b-1]$. Since $y=ax^c+b$ has no solution in $[a+b-1]$, it follows that there is no monochromatic solution of $y=ax^c+b$ in $[a+b-1]$. Under this coloring, there is no rainbow solution of $y=x+2$. Thus, $\operatorname{GR}_{2}\left(y=x+2:y=ax^c+b\right)\ge a+b$.
	
For the upper bound, we consider any red/blue-coloring of $[a+b]$. To the contrary, suppose that there is a red/blue-coloring of $[a+b]$ such that there is neither a rainbow solution of $y=x+2$ nor a monochromatic solution of $y=ax^c+b$. Since there is no rainbow solution of $y=x+2$, it follows from Corollary~\ref{Cor-Colored structure} that without considering the order of colors, the coloring is uniquely determined. Without loss of generality, we assume that all odd numbers red and all even numbers blue in $[a+b]$. Noticing that due to the different parity of $a$ and $b$, $a+b$ is odd. In this case, $(1, a+b)$ is a monochromatic solution of $y=ax^c+b$, which is a contradiction. The result thus follows.

\begin{case}
$a$ and $b$ are even.
\end{case}
For the lower bound, we color all odd numbers red and all even numbers blue in $[a\cdot 2^c+b-1]$. Noticing that $y=ax^c+b$ has only one solution in $[a\cdot 2^c+b-1]$, which is $(1, a+b)$. But since both $a$ and $b$ are even, $a+b$ is also even, that is, the solution $(1, a+b)$ is not a monochromatic solution. Also, under this coloring, there is no rainbow solution of $y=x+2$. Thus, $\operatorname{GR}_{2}\left(y=x+2:y=ax^c+b\right)\ge a\cdot 2^c+b$.
	
For the upper bound, we consider any red/blue-coloring of $[a\cdot 2^c+b]$. To the contrary, suppose that there is a red/blue-coloring of $[a\cdot 2^c+b]$ such that there is neither a rainbow solution of $y=x+2$ nor a monochromatic solution of $y=ax^c+b$. Since there is no rainbow solution of $y=x+2$, it follows from Corollary~\ref{Cor-Colored structure} that without considering the order of colors, the coloring is uniquely determined. Without loss of generality, we assume that all odd numbers red and all even numbers blue in $[a\cdot 2^c+b]$. Since both $a$ and $b$ are even, it follows that $a\cdot 2^c+b$ is also even. In this case, $(2, a\cdot 2^c+b)$ is a monochromatic solution of $y=ax^c+b$, which is a contradiction. The result thus follows.
\end{proof}

It follows from Corollary~\ref{Cor-Colored structure} that a $b$-colored $\mathbb{N}_{+}$ contains no rainbow solution of $y=x+b$ if and only if for each $\lambda\in [b]$, the $\lambda$-class $\mathcal{C}_{(y=x+b,\lambda)}\subseteq \mathbb{N}_{+}$ is monochromatic. For example, $\mathcal{C}_{(y=x+2,1)}$ is the set of all positive odd numbers, and $\mathcal{C}_{(y=x+2,2)}$ is the set of all positive even numbers. Since $\mathcal{C}_{(y=x+2,1)}$ and $\mathcal{C}_{(y=x+2,2)}$ are arithmetic sequences, it follows that $(x_0, ax_0^c+b)$ is a solution of $y=ax^c+b$ in $\mathcal{C}_{(y=x+2,1)}$ or $\mathcal{C}_{(y=x+2,2)}$ if and only if $\frac{ax_0^c+b-x_0}{2}$ is an integer. If $a$ and $b$ are odd, then $\frac{ax_0^c+b-x_0}{2}$ is not an integer for any $x_0\in\mathbb{N}_{+}$, which implies that there is no monochromatic solution of $y=ax^c+b$ in $\mathbb{N}_{+}$. If $a$ and $b$ have different parity, then $\frac{a\cdot 1^c+b-1}{2}$ is an integer, which implies that one of the monochromatic solution of $y=ax^c+b$ in $\mathbb{N}_{+}$ is $(1,a+b)$. If $a$ and $b$ are even, then $\frac{ax_0^c+b-x_0}{2}$ is not an integer for $x_0=1$ but is an integer for $x_0=2$, which implies that one of the monochromatic solution of $y=ax^c+b$ in $\mathbb{N}_{+}$ is $(2,a\cdot 2^c+b)$. Based on the above ideas, we provide the following result.

\begin{theorem}
For a strictly monotonically increasing binary function equation $y=f(x)$ such that $f(x)\in \mathbb{N}_+$ for all $x\in\mathbb{N}_+$, and integer $b\ge 2$, we have
	\begin{equation*}
		\operatorname{GR}_{b}\left(y=x+b:y=f(x)\right)=f(x_{\min}),
	\end{equation*}
	where $x_{\min}=\min\left\{x\in\mathbb{N}_{+}: \text{$\frac{f(x)-x}{b}$ is an integer}\right\}$.  Moreover, if $x_{\min}$ does not exist, then $\operatorname{GR}_{b}\left(y=x+b:y=f(x)\right)$ also does not exist.
\end{theorem}
\begin{proof}
Let the $\lambda$-class $\mathcal{C}_{(y=x+b,\lambda)}=\{\lambda,\lambda+b,\lambda+2b,\ldots\}\subseteq \mathbb{N}_{+}$ for each $\lambda\in [b]$, and the coloring $\chi$ of $\mathbb{N}_{+}$ is as described in Corollary~\ref{Cor-Colored structure}. Specifically, the $\lambda$-class $\mathcal{C}_{(y=x+b,\lambda)}$ is monochromatic for each $\lambda\in [b]$, and for different $\lambda_i$ and $\lambda_j$ in $[b]$, $\mathcal{C}_{(y=x+b,\lambda_i)}$ and $\mathcal{C}_{(y=x+b,\lambda_j)}$ have different colors. Obviously, under the coloring $\chi$, $y=x+b$ does not have a rainbow solution in $\mathbb{N}_{+}$. Noticing that for each $\lambda\in [b]$, $\mathcal{C}_{(y=x+b,\lambda)}$ is an arithmetic sequence. Therefore, $(x_0, f(x_0))$ is a solution of $y=ax^c+b$ in $\mathcal{C}_{(y=x+b,\lambda)}$ for some $\lambda\in [b]$ if and only if $\frac{f(x_0)-x_0}{b}$ is an integer. Let $x_{\min}=\min\left\{x\in\mathbb{N}_{+}: \text{$\frac{f(x)-x}{b}$ is an integer}\right\}$. If $x_{\min}$ does not exist, then $\frac{f(x_0)-x_0}{b}$ is not an integer for any $x_0\in\mathbb{N}_{+}$, which implies that there is no monochromatic solution of $y=ax^c+b$ in $\mathbb{N}_{+}$, and thus $\operatorname{GR}_{b}\left(y=x+b:y=f(x)\right)$ does not exist.
	
Next, we assume that $x_{\min}$ exists. For the lower bound, we apply the coloring $\chi$ to $\mathcal{C}_{(y=x+b,\lambda)}\subseteq [f(x_{\min})-1]$. If there is a monochromatic solution of $y=f(x)$ in $[f(x_{\min})-1]$, then it contradicts the minimality of $x_{\min}$. Therefore, under the coloring $\chi$, there is neither a rainbow solution of $y=x+b$ nor a monochromatic solution of $y=f(x)$ in $[f(x_{\min})-1]$. For the upper bound, we consider any $b$-coloring of $[f(x_{\min})]$. To the contrary, suppose that there is a $b$-coloring of $[f(x_{\min})]$ such that there is neither a rainbow solution of $y=x+b$ nor a monochromatic solution of $y=f(x)$. Since there is no rainbow solution of $y=x+b$, it follows from Corollary~\ref{Cor-Colored structure} that without considering the order of colors, the coloring $\chi$ is uniquely determined. But under the coloring $\chi$, we can find a monochromatic solution of $y=f(x)$ in $[f(x_{\min})]$, which is $(x_{\min},f(x_{\min}))$, a contradiction. The result thus follows.
\end{proof}

\section{Conclusion}
When studying the Gallai--Rado number $\operatorname{GR}_{k}(\mathcal{E}_1:\mathcal{E}_2)$, it is important that there is no rainbow solution of the equation $\mathcal{E}_1$ in $[n]$ with a specific colored structure. The rainbow number of the equation $\mathcal{E}_1$ provides some valuable information. We can know with at least how many colors are needed to color $[n]$, and there always be a rainbow solution of the equation $\mathcal{E}_1$ in $[n]$. For example, it follows from $\operatorname{rb}([n],y=x+b)=b+1$ that for colored $[n]$ with $b+1$ colors, there always be a rainbow solution of the equation $y=x+b$ in $[n]$, while with $b$ colors, there may not be a rainbow solution of the equation $y=x+b$ in $[n]$. 

Noticing that the value of $\operatorname{rb}([n],y=x+b)$ dose not rely on $n$, so no matter how large $n$ is, as long as $[n]$ is colored with $b$ colors, its colored structure without rainbow solution of the equation $y=x+b$ is uniquely determined (see Corollary~\ref{Cor-Colored structure}). However, for the general linear binary equation $y=ax+b$, where $a\ge 2, b\ge 0$, the rainbow number $\operatorname{rb}([n],y=ax+b)$ rely on $n$, which leads to the fact that for different $n$, the maximum number of colors required to satisfy the equation $y=ax+b$ without rainbow solution in $[n]$ is also different. It can be seen that if $\operatorname{rb}([n],\mathcal{E})$ does not rely on $n$, then the study of the Gallai--Rado number involving rainbow $\mathcal{E}$ becomes much easier. A natural question is whether there exists a nonlinear binary function equation $y=f(x)$ whose rainbow number does not rely on $n$?

Since the number of colors of the Gallai--Rado numbers studied in this paper is fixed at $b$, where the rainbow equation is $y=x+b$, the following problems can be considered. When the number of colors $k$ satisfies $2\le k\le b-1$, for different equations $\mathcal{E}$, study the Gallai--Rado numbers $\operatorname{GR}_{k}(y=x+b:\mathcal{E})$. In fact, when $2\le k\le b-1$, the difficulty lies in the fact that the colored structure of $[n]$ that does not have rainbow solution of the equation $y=x+b$ is not uniquely determined. On the other hand, we can change the rainbow equation $y=x+b$ studied in this paper to study other rainbow equations, such as studying $\operatorname{GR}_{k}(y=ax:\mathcal{E})$ with a given number of colors $k$ and a positive integer $a\ge 2$.

\end{document}